\documentclass{llncs}
\usepackage{makeidx}  
\usepackage{amssymb,amsmath}

\begin{document}

\frontmatter          
\addtocmark{Weak abelian periodicity of infinite words} 

\title{Weak abelian periodicity of infinite words}
\titlerunning{Weak abelian periodicity of infinite words}  
%
%
\author{Sergey Avgustinovich\inst{1} \and Svetlana Puzynina\inst{1,2}\thanks{Supported in part by the Academy of Finland under grant 251371,
by Russian Foundation of Basic Research (grant 12-01-00448), and
by RF President grant MK-4075.2012.1.} }
\authorrunning{S. Avgustinovich, S. Puzynina} 
%
\tocauthor{Sergey Avgustinovich, Svetlana Puzynina}
\institute{Sobolev Institute of Mathematics, Russia,
\email{avgust@math.nsc.ru} \and University of Turku, Finland,
\email{svepuz@utu.fi}}

\maketitle              

\begin{abstract}
We say that an infinite word $w$ is \emph{weak abelian periodic}
if it can be factorized into finite words with the same
frequencies of letters. In the paper we study properties of weak
abelian periodicity, its relations with balance and frequency. We
establish necessary and sufficient conditions for weak abelian
periodicity of fixed points of uniform binary morphisms. Also, we
discuss weak abelian periodicity in minimal subshifts.
\end{abstract}
The study of abelian properties of words dates back to Erd\"{o}s's
question whether there is an infinite word avoiding abelian
squares \cite{e}. Abelian powers and their avoidability in
infinite words is a natural generalization of analogous questions
for ordinary powers. The answer to Erd\"{o}s's question has been
given by Ker\"{a}nen, who provided a construction of an abelian
square-free word \cite{keranen}. From that time till nowadays,
many problems concerning different abelian properties of words
have been studied, including abelian periods, abelian powers,
avoidability, complexity (see, e. g., \cite {CI}, \cite {crsz},
\cite {cft}, \cite {pz}).

Two words are said to be abelian equivalent, if they are
permutations of each other. Simirarly to usual powers, an abelian
$k$-power is a concatenation of $k$ abelian equivalent words. We
define a \emph{weak abelian power} as a concatenation of words
with the same frequencies of letters. So, in a weak abelian power
we admit words with different lengths; if all words are of the
same length, then we have an abelian power. Earlier some questions
about avoidability of weak abelian powers have been considered. In
\cite{K} for given integer $k$ the author finds an upper bound for
length of binary word which does not contain weak abelian
$k$-powers. In \cite{gr} the authors build an infinite ternary
word having no weak abelian $(5^{11}+1)$-powers.

The notion of abelian period is a generalization of the notion of
normal period, and it is closely related with abelian powers. A
periodic infinite word can be defined as an infinite power.
Similarly, we say that a word is (weak) abelian periodic, if it is
a (weak) abelian $\infty$-power. In the paper we study the
property of weak abelian periodicity for infinite words, in
particular, its connections with related notions of balance and
frequency. We establish necessary and sufficient conditions for
weak abelian periodicity of fixed points of uniform binary
morphisms. Also, we discuss weak abelian periodicity in minimal
subshifts.

The paper is organized as follows. In Section 2 we fix our
terminology, in Section 3 we discuss some general properties of
weak abelian periodicity and its connections with other notions,
such as balance and frequencies of letters. In Section 4 we give a
criteria for weak abelian periodicity of fixed points of primitive
binary uniform morphisms. In Section 5 we study weak abelian
periodicity of points in shift orbit closure of uniform recurrent
words. 

\section{Preliminaries}
In this section we give some basics on words following terminology
from \cite {Lo} and introduce our notions.

Given a finite non-empty set $\Sigma$ (called the alphabet), we
denote by $\Sigma^*$ and $\Sigma^{\omega}$, respectively, the set
of finite words and the set of (right) infinite words over the
alphabet $\Sigma$. Given a finite word $u = u_1 u_2 \dots u_n$
with $n \geq 1$ and $u_i \in \Sigma$, we denote the length $n$ of
$u$ by $|u|$. The empty word will be denoted by $\varepsilon$ and
we set $|\varepsilon| = 0$.

Given the words $w$, $x$, $y$, $z$ such that $w = xyz$, $x$ is
called a \emph{prefix}, $y$ is a \emph{factor} and $z$ a
\emph{suffix} of $w$. The factor of $w$ starting at position $i$
and ending at position $j$ will be denoted by $w[i, j] =
w_iw_{i+1} \dots w_j$. The prefix (resp., suffix) of length $n$ of
$w$ is denoted
${\rm pref}_n(w)$ (resp., ${\rm suff}_n(w)$). 
 The set of all factors of $w$ is denoted by
$F(w)$, the set of all factors of length $n$ of $w$ is denoted by
$F_n(w)$.

An infinite word $w$ is \emph{ultimately periodic}, if for some
finite words $u$ and $v$ it holds $w=uv^{\omega}$; $w$ is
\emph{purely periodic} (or briefly \emph{periodic}) if
$u=\varepsilon$. An infinite word is \emph{aperiodic} if it is not
ultimately periodic.

An infinite word $w=w_1w_2\dots$ is \emph{recurrent} if any of its
factors occurs infinitely many times in it. The word $w$ is
\emph{uniformly recurrent} if any its factor $u$ there exists $C$
such that whenever $w[i,j]=u$, there exists $0<k\leq C$ such that
$w[i,j]=w[i+C,j+C]=u$. In other words, factors occur in $w$ in a
bounded gap.

Given a finite word $u = u_1 u_2 \dots u_n$ with $n \geq 1$ and
$u_i \in \Sigma$, for each $a \in \Sigma$, we let $|u|_a$ denote
the number of occurrences of the letter $a$ in $u$. Two words $u$
and $v$ in $\Sigma^*$ are \emph{abelian equivalent}, denoted $u
\sim_{ab} v$, if and only if $|u|_a = |v|_a$ for all $a
\in \Sigma$. It is easy to see that abelian equivalence is indeed 
an equivalence relation on $\Sigma^*$.

An infinite word $w$ is called \emph{abelian (ultimately)
periodic}, if $w = v_0 v_1 \dots$, where $v_k\in\Sigma^*$ for
$k\geq0$, and $v_i \sim_{ab} v_j$ for all integers $i \geq 1$, $j
\geq 1$.

For a finite word $w\in\Sigma^*$, we define \emph{frequency}
$\rho_a(w)$ of a letter $a\in\Sigma$ in $w$ as
$\rho_a(w)=\frac{|w|_a}{|w|}$.

\begin{definition} An infinite word $w$ is called \emph{weak abelian
(ultimately) periodic}, if $w = v_0 v_1 \dots$, where $v_i\in
\Sigma^*$, $\rho_a(v_i) = \rho_a(v_j)$ for all $a\in\Sigma$ and
all integers $i, j\geq 1$.
\end{definition}

In other words, a word is weak abelian periodic if it can be
factorized into words of different lengths with the same
frequencies of letters. In the further text we usually omit the
word ``ultimately'', meaning that there can be a prefix with
different frequencies. Also, we often write WAP instead of weak
abelian periodic for brevity.

\begin{definition} An infinite word $w$ is called \emph{bounded weak abelian
periodic}, if it is weak abelian periodic with bounded lengths of
blocks, i. e., there exists $C$ such that for every $i$ we have
$|v_i|\leq C$. \end{definition}

We mainly focus on binary words, but we also make some
observations in the case of general alphabet. One can consider the
following geometric interpretation of weak abelian periodicity.
Let $w=w_1 w_2 \dots$ be an infinite word over a finite alphabet
${\rm \Sigma}$. We translate $w$ to a graphic visiting points of
the infinite rectangular grid by interpreting letters of $w$ by
drawing instructions. In the binary case, we assign $0$ with a
move by vector $\textbf{v}_0=(1,-1)$, and $1$ with a move $\textbf{v}_1=(1,1)$. We
start at the origin $(x_0,y_0)=(0,0)$. At step $n$, we are at a
point $(x_{n-1}, y_{n-1})$ and we move by a vector corresponding
to the letter $w_{n}$, so that we come to a point $(x_{n},
y_{n})=(x_{n-1}, y_{n-1})+v_{w_n}$, and the two points $(x_{n-1},
y_{n-1})$ and $(x_{n}, y_{n})$ are connected with a line segment.
So, we translate the word $w$ to a path in $\mathbb{Z}^2$. We
denote corresponding graphic by $g_w$. So, for any word $w$, its
graphic is a piece-wise linear function with linear segments
connecting integer points. It is easy to see that weak abelian
periodic word $w$ has graphic with infinitely many integer points
on a line with rational slope (we will sometimes write that $w$ is
WAP along this line). A bounded weak abelian periodic word has a
graphic with bounded differences between letters. Note also that
instead of vectors $(1,-1)$ and $(1,1)$ one can use any other pair
of noncollinear vectors $\textbf{v}_0$ and $\textbf{v}_1$, and sometimes it will be
convenient for us to do it. For a $k$-letter alphabet one can
consider a similar graphic in $\mathbb{Z}^k$. Note that the
graphic can also be defined for finite words in a similar way, and
we will sometimes use it.

\begin{definition} We say that a word $w$ is of \emph{bounded
width}, if there exist two lines with the same rational slope, so
that the path corresponding to $w$ lies between these two lines.
Formally, there exist rational numbers $a, b_1, b_2$, so that
$ax+b_1 \leq g_w(x)\leq ax+b_2$. \end{definition}

Note that we focus on rational $a$, because words of bounded
irrational width cannot be weak abelian periodic. Equivalently,
bounded width means that graphic of the word lies on finitely many
lines with rational coefficients.

We will also need notions of frequency and balance, which are
closely connected with abelian periodicity. Relations between
these notions are discussed in the next section. A word $w$ is
called \emph{$C$-balanced} if for any its factors $u$ and $v$ of
equal length $| |u|_a - |v|_a |\leq C$ for any $a\in\Sigma$.
Actually, the notion of bounded width is equiivalent to the notion
of balance (see, e.g., \cite {a}). We say that a letter
$a\in\Sigma$ has \emph{frequency} $\rho_a(w)$ in $w$ if
$\rho_a(w)=\lim_{n\to \infty}\rho_a({\rm pref}_n(w))$. Note that
for some words the limit does not exist, and we say that such
words do not have letter frequencies. Note also that we define
here a prefix frequency, though sometimes another version of
frequency of letters in words is studied (see Section 5 for
definitions). Remark that if a WAP word has a frequency of a
letter, then this frequency coincides with frequency of this
letter in factors of corresponding factorization.

A \emph{morphism} is a function $\varphi : \Sigma^*\to \Delta^*$
such that $\varphi(\varepsilon) = \varepsilon$ and $\varphi(uv) =
\varphi(u)\varphi(v)$, for all $u, v \in \Sigma^*$. Clearly, a
morphism is completely defined by the images of the letters in the
domain. For most of morphisms we consider, $\Sigma = \Delta$. A
morphism is \emph{primitive}, if there exists $k$ such that for
every $a\in\Sigma$ the image $\varphi^k(a)$ contains all letters
from $\Delta$. A morphism is \emph{uniform}, if
$|\varphi(a)|=|\varphi(b)|$ for all $a,b\in\Sigma$, and
prolongeable on $a \in \Sigma$, if $a={\rm pref}_1(\varphi (a))$.
If $\varphi $ is prolongeable on $a$, then $\varphi^n(a)$ is a
proper prefix of $\varphi^{n+1}(a)$, for all $n \in \omega$.
Therefore, the sequence $(\varphi^n(a))_{n\geq 0}$ of words
defines an infinite word $w$ that is a fixed point of $\varphi$.

\medskip

Remind the definition of Toeplitz words. Let ? be a letter not in
$\Sigma$ . For a word $w \in \Sigma ( \Sigma \cup ? )^*$, let
$$T_0 ( w ) = ? ^{\omega}  , T_{i+1}(w)=F_w (T_i(w)), $$
where $F_w ( u )$, defined for any $u \in ( \Sigma  \cup  ?
)^{\omega}$, is the word obtained from $w^{\omega}$ by replacing
the sequence of all occurrences of ? by $u$; in particular, $F_w (
u ) = w^{\omega}$ if $w$ contains no ?.

Clearly,
$$T ( w ) = \lim_{i\to\infty} T_i ( w ) \in \Sigma^{\omega} $$
is well-defined , and it is referred to as the \emph{Toeplitz word}
determined by the pattern $w$. Let $p = | w |$ and $q = | w |_?$
be the length of $w$ and the number of ?'s in $w$, respectively.
Then  $T ( w )$ is called a $( p , q )$-Toeplitz word.

\medskip

\noindent \textbf{Example 1}. Paperfolding word: $$00100110001101100010011100110110\dots$$
This word can be defined, e.g., as a Toeplitz word with pattern $w =
0?1?$. The graphic corresponding to the paperfolding
word with $\textbf{v}_0=(1,-1)$, $\textbf{v}_1=(1,1)$ is in Fig. 1. The paperfolding
word is not balanced and is WAP along the line $y=-1$ (and actully
along any line $y=C$, $C=-1, -2, \dots$). See Proposition
\ref{boundedwidth} (2) for details.

\medskip

\begin{picture}(330,80)
\multiput(10,0)(10,0){32}{\line(0,1){80}}
\multiput(0,10)(0,10){7}{\line(1,0){330}}

\thicklines \put(0,70){\vector(1,0){330}}
\put(10,0){\vector(0,1){80}} \put(10,70){\circle*{5}}

\put(10,70){\line(1,-1){20}} \put(30,50){\line(1,1){10}}
\put(40,60){\line(1,-1){20}} \put(60,40){\line(1,1){20}}
\put(80,60){\line(1,-1){30}} \put(110,30){\line(1,1){20}}
\put(130,50){\line(1,-1){10}} \put(140,40){\line(1,1){20}}
\put(160,60){\line(1,-1){30}} \put(190,30){\line(1,1){10}}
\put(200,40){\line(1,-1){20}} \put(220,20){\line(1,1){30}}
\put(250,50){\line(1,-1){20}} \put(270,30){\line(1,1){20}}
\put(290,50){\line(1,-1){10}} \put(300,40){\line(1,1){20}}
\put(320,60){\line(1,-1){10}}

\end{picture}

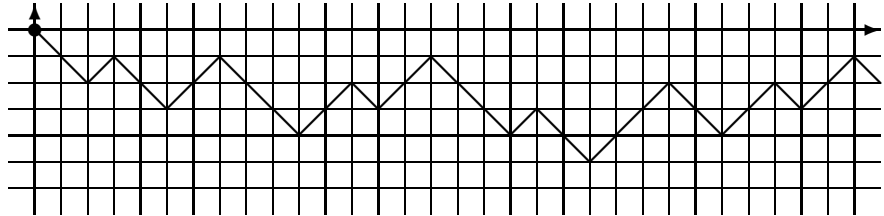
\vskip18pt \small \noindent\parbox[t]{162mm}{{\bf Fig. 1.} The
graphic of the paperfolding word with $\textbf{v}_0=(1,-1)$,
$\textbf{v}_1=(1,1)$.} \normalsize

\bigskip

\noindent \textbf{Example 2}. A word obtained as an image of the morphism $0 \to 01$, $1\to 0011$ of any nonperiodic binary word is bounded WAP.

\section{General properties of weak abelian periodicity}

In this section we discuss relations between notions defined in
the previous section and observe some simple properties of weak
abelian periodicity. We start with the property of bounded width
and its connections to weak abelian periodicity.

\begin{proposition}

\noindent 1. If an infinite word $w$ is of bounded width, then $w$
is WAP.

\noindent 2. There exists an infinite word $w$ of bounded width
which is not bounded WAP.

\noindent 3. If an infinite word $w$ is bounded WAP, then $w$ is
of bounded width.

\end{proposition}

\begin{proof}

1. Since $w$ is of bounded width, its graphic lies on a finite
number of lines with rational coefficients. By the pigeonhole
principle it has infinitely many points on one of these lines and
hence is WAP.

\medskip

2. Consider $$w=01 1 1010 0 010101 1 10101010 \dots = (01)^1 1
(10)^2 0 (01)^3 1 (10)^4 \dots (01)^{2i-1} 1 (10)^{2i} 0 \dots$$
Taking its graphic with $\textbf{v}_0=(-1,1)$ and $\textbf{v}_1=(1,1)$ we see that
it lies on the lines $y=0, -1, 1, 2$ and hence $w$ is of bounded
width. The graphic intersects each of these lines infinitely many
times, but each of them with growing gaps.

\medskip

3. Again, take graphic of $w$ with $\textbf{v}_0=(-1,1)$ and $\textbf{v}_1=(1,1)$.
 Bounded WAP means that it intersects
some line $y=ax+b$ with $a$, $b$ rational and gap at most $C$ for
some integer $C$, i. e., the difference between two consecutive
points $x_i$ and $x_{i+1}$ is at most $C$. So, the graphic lies
between lines $y=ax+b-C/2$ and $y=ax+b +C/2$, and hence $w$ is of
bounded width.

\end{proof}

In the following proposition we discuss the connections between uniform
recurrence and WAP.

\begin{proposition}\label{boundedwidth}

\noindent 1. If $w$ is uniformly recurrent and of bounded width,
then $w$ is bounded WAP.

\noindent 2. There exists a uniformly recurrent WAP word $w$ which
is not of bounded width.

\end{proposition}

\begin{proof}
1. Take graphic of $w$ with some vectors, e. g., $\textbf{v}_0=(-1,1)$ and
$\textbf{v}_1=(1,1)$. Bounded width means that the graphic $g_w$ satisfies
$ax+b_1 \leq g_w(x)\leq ax+b_2$ for some rational numbers $a, b_1,
b_2$, so that $ax+b_1 \leq g_w(x)\leq ax+b_2$. Take the biggest
such $b_1$ and the smallest $b_2$, i. e., there are integers $x_1$
and $x_2$ such that $g_w({x_1})=ax_1+b_1$, $g_w({x_2})=ax_2+b_2$.
Without loss of generality suppose $x_1\leq x_2$ and consider the
factor $w[x_1,x_2]$. Since $w$ is uniformly recurrent, this factor
occurs infinitely many times in it with bounded gap. Every
position $i$
 corresponding to an occurrence of this factor satisfies $g_w(i) =
 ai+b_1$, otherwise $g_w(i+x_2-x_1) > a (i+x_2-x_1) +b_2$, which
 contradicts the choice of $b_2$. Hence the word is bounded WAP
 along the line $y=ax+b_1$ (and moreover along $y=ax+b_2$ and any rational line in
 between).

\medskip

 2. One of such examples is the paperfolding word $w$. It can be
 defined in several equivalent ways, we define it as a Toeplitz word
 with pattern $0?1?$ \cite {CaK}. It is not difficult to see that $|{\rm pref}_{4^k-1}(w)|_0=4^{k}/2$, $|{\rm
 pref}_{4^k-1}(w)|_1=4^{k}/2-1$. So, the word is WAP with frequencies $\rho_0=\rho_1=\frac{1}{2}$ along
 the line $y=-1$. On the other hand, taking $n=2^k+2^{k-2}+ \dots +2^{k-2\lfloor\frac{k}{2}\rfloor}$,
 one gets $|{\rm pref}_{n}(w)|_0-|{\rm pref}_{n}(w)|_1=k+1$. So, the
 word is not of bounded width. \end{proof}

Next, we study the relation between WAP property and frequencies
of letters.

\begin{proposition} \label{freq}

\noindent 1. There exists an infinite word $w$ with rational
frequencies of letters which is not WAP.

\noindent 2. If an infinite word $w$ has irrational frequency of
some letters, then $w$ is not WAP.

\noindent 3. If a binary infinite word $w$ does not have
frequencies of letters, then $w$ is WAP.

\noindent 4. There exist a ternary infinite word $w$ which is does
not have frequencies of letters and which is not WAP.

\end{proposition}

\begin{proof}

1. Consider $$w=01 0 0101 0 (01)^4 \dots 0 (01)^{2^n} \dots$$ This
word has letter frequencies $\rho_0=\rho_1=1/2$. Suppose it is
weak abelian periodic. If a word has frequencies of letters and is
WAP, then these frequencies coincide with frequencies of letters
in the corresponding factorization. So, if $w$ is WAP, then there
is a sequence $k_1, k_2, \dots$ (the sequence of lengths of
factors in the corresponding factorization), such that $|{\rm
pref}_{k_i} w|_0=k_i/2+C$, where $C$ is defined by the first
factor of length $k_1$: $C=k_1/2-|{\rm pref}_{k_1} w|_0/2$. For
the word $w$, the number of $0$-s in a prefix of length $n$ is
$|{\rm pref}_{n} w|_0 = n/2+\theta (\log n)$. For $n=k_i$ large
enough one has $\theta (\log n)>C$, a contradiction. So, $w$ is
not WAP.

For uniformly recurrent examples see Section 5.

\medskip

2. Assume that the word $w$ is WAP, then for every letter $a$
there exists a rational partial limit $\lim_{n_k\to \infty}
\frac{|{\rm pref}_{n_k} (w)|_a}{|{\rm pref}_{n_k} (w)|}$. For $w$
having irrational frequency of some letter all such partial limits
corresponding to this letter exist and are equal to this
irrational frequency. A contradiction.

\medskip

3. Consider a sequence $(\frac{|{\rm pref}_{n} (w)|_a}{|{\rm
pref}_{n} (w)|})_{n\geq 1}$. This sequence is bounded, and has a
lower and upper partial limits
$r=\underline{\lim}_{n\to\infty}\frac{|{\rm pref}_{n}
(w)|_a}{|{\rm pref}_{n} (w)|} $ and $R=
\overline{\lim}_{n\to\infty}\frac{|{\rm pref}_{n} (w)|_a}{|{\rm
pref}_{n} (w)|}$. Since the sequence does not have a limit, these
partial limits do not coincide: $r<R$. Using graphic of $w$, one
gets that the graphic intersects every line with slope
corresponding to frequency between $r$ and $R$. For rational
frequencies one gets that the graphic intersects the line
infinitely many times. Hence there are infinitely many integer
points on it (or its shift, depending on the choice of $v_0$ and
$v_1$). So, we proved that $w$ is WAP, and moreover, it is WAP
with any rational frequency $\rho$, $r<\rho<R$ in factors in
corresponding factorization.

\medskip

4. Consider a word $$w=012 01^2 2^4 0^6 1^{10} 2^{16} \dots
0^{n_i}1^{n_{i+1}}2^{n_{i+2}} \dots ,$$ where
$n_i={n_{i-1}+n_{i-2}}$ for every $i\geq 5$, $n_1=n_2=n_3=n_4=1$.
The word is organized in a way that after each block $a^{n_i}$ the
frequency of the letter $a$ in the prefix ending in this block is
equal to $1/2$, i. e., $\rho_a(012 0 1^2 2^4 \dots
a^{n_i})=\frac{1}{2}$ for $a \in \{0, 1, 2\}$. So, frequencies of
letters do not exist.

Now we will prove that it is not weak abelian periodic. Suppose it
is, with points $k_1, k_2 \dots $ and rational frequencies
$\rho_0, \rho_1, \rho_2$ in the blocks, i. e. $w=w_1 w_2 \dots$,
and $|w_1 \dots w_n|=k_n$ and $\frac {|w_i|_a}{|w_i|}=\rho_a$ for
every $a\in \{0,1,2\}$ and $i>1$. By the pigeonhole principle
there exists a letter $a$ such that infinitely many $k_i$ are in
the blocks of $a$-s, meaning that at least one of the letters
$w_{k_i}$, $w_{k_{i}+1}$ is $a$. Without loss of generality
suppose $a=2$. Using the recurrence relation for $n_i$, one can
find $\lim_{n\to \infty} \frac {|{\rm pref}_{k_n} w|_0} {|{\rm
pref}_{k_n} w|_1}=\frac{1}{\lambda_1}$, where
$\lambda_1=\frac{1+\sqrt{5}}{2}$ is the larger root of the
equation $\lambda^2=\lambda+1$ corresponding to the recurrence
relation. So, the limit is irrational, and hence $w$ cannot be
equal to $\frac {\rho_0}{\rho_1}$. Thus, $w$ is not WAP.
\end{proof}

So, we obtain the following corollary:

\begin{corollary}
If a binary word $w$ is not WAP, then it has frequencies of
letters.
\end{corollary}

This simple corollary, however, is unexpected: from the first
glance weak abelian periodicity and frequencies of letters seem to
be very close notions. But it turns out that one of them (WAP)
does not hold, then the other one should necessarily hold.

\medskip

We end this section with an observation about WAP of non-binary words. We will show that contrary to normal and abelian periodicity, the property WAP cannot be checked from binary words obtained by unifying letters of the original word. 

For a word $w$ over an alphabet of cardinality $k$ define
$w^{a\cup b}$ as a word over an alphabet of cardinality $k-1$
obtained form $w$ by unifying letters $a$ and $b$. In other words,
$w^{a\cup b}$ is an image of $w$ under a morphism $b\to a$, $c\to
c$ for every $c\neq b$.

\begin{proposition}
There exists a ternary word $w$, such that $w^{0\cup 1}$,
$w^{0\cup 2}$, $w^{1\cup 2}$ are WAP, and $w$ itself is not WAP.
\end{proposition}

\begin{proof} We use the example we built in the proof of Proposition \ref{freq} (3), i. e., we take
$w=012 01^2 2^4 0^6 1^{10} 2^{16} \dots
0^{n_i}1^{n_{i+1}}2^{n_{i+2}} \dots $, where
$n_i={n_{i-1}+n_{i-2}}$ for every $i$. Due to space limitations,
we omit the calculations.
\end{proof}

\section{Weak abelian periodicity of fixed points of binary uniform morphisms}

In this section we study weak abelian periodicity of fixed points
of non-primitive uniform binary morphisms.

Consider a binary uniform morphism $\varphi$ with matrix $\left(
\begin{array}{cc}
a & b \\
c & d \end{array} \right)$. This means that $|\varphi(0)|_0=a$,
$|\varphi(0)|_1=b$, $|\varphi(1)|_0=c$, $|\varphi(1)|_1=d$, and
$a+b=c+d=k$, since we consider a uniform morphism. In a fixed
point $w$ of the binary uniform morphism $\varphi$ the frequencies
exist and they are rational. It is easy to see that
$\rho_0(w)=\frac{c}{b+c}$, $\rho_1(w)=\frac{b}{b+c}$. It will be
convenient for us to consider a geometric interpretation with
$\textbf{v}_0=(1,-b)$, $\textbf{v}_1=(1,c)$. If $w$ is WAP, then the frequency
inside the blocks is equal to the frequency in the whole word. So,
WAP can be reached along a horizontal line $y=C$.

The following theorem gives a characterization of weak abelian
periodicity for fixed points of non-primitive binary uniform
morphisms:

\begin{theorem} \label{morph}

Consider a non-primitive binary uniform morphism $\varphi$ with
matrix $\left(
\begin{array}{cc}
a & b \\
c & d \end{array} \right)$ having a fixed point $w$ starting in
$0$. For any $u\in \{0, 1\}^*\cup \{0, 1\}^{\infty}$ let $g_u$ be
its graphic with vectors $\textbf{v}_0=(1,-b)$, $\textbf{v}_1=(1,c)$.

\noindent 1. If $g_{\varphi (0)}(x)=0$ for some $x$, $0< x \leq
k$, then $w$ is WAP.

\noindent 2. If $g_{\varphi (0)}(k)\geq -b$, then $w$ is WAP.

\noindent 3. Otherwise we need the following parameters. Denote
$\Delta = g_{\varphi (0)}(k)$, $A = \max \{g_{\varphi (0)}(i) |
i=1, \dots k, w_i=1 \}$, $t=\max \{g_{\varphi (1)}(i) | i=1, \dots
k, w_i=1 \}$.

If $\varphi$ does not satisfy conditions 1 and 2, then its fixed
point $w$ is WAP if and only if $\Delta\frac{A-c}{-b}+t\geq A$.

\end{theorem}


\begin{proof} 

1. If in the condition $g_{\varphi (0)}(x)=0$, $0< x \leq k$, the
number $x$ is integer, then for every $i$ it holds
$g_{\varphi^i(0)}(k^{i-1}x) = 0$, so the word is WAP. If $x$ is
not integer, then we have either $g_{\varphi (0)}(\lfloor x
\rfloor)<0$ and  $g_{\varphi (0)}(\lceil x \rceil)>0$ or
$g_{\varphi (0)}(\lfloor x \rfloor)>0$ and  $g_{\varphi
(0)}(\lceil x \rceil)<0$. Without loss of generality consider the
first case. For any $i$, one has $g_{\varphi^i (0)}(k^{i-1}\lfloor
x \rfloor)<0$ and  $g_{\varphi^i (0)}(k^{i-1}\lceil x \rceil)>0$,
so there exists $x_i$, $k^{i-1}\lfloor x \rfloor < x_i <
k^{i-1}\lceil x \rceil$, such that $g_{\varphi^i (0)}(x_i)=0$. So,
we have an infinite sequence of points $(x_i)_{i=1}^{\infty}$ such
that $g_w(x_i)=0$. By the definition of $g_w$ and the pigeonhole
principle
 we get that there is an infinite number of integer points from the set $\lfloor x_i\rfloor, \lceil x_i \rceil$, $i=1, \dots, \infty$, on one of the
 lines $x=A$, $A=-\max(b,c)+1, -\max(b,c)+2, \dots, \max(b,c)-1$.
 So, $w$ is WAP.

\medskip

2. If $g_{\varphi (0)}(k)\geq 0$, we are in the conditions of the
case 1, so the word is WAP. If $0>g_{\varphi (0)}(k)\geq -b$, then
the only possible case is $g_{\varphi (0)}(k)=-b$. This follows
from the fact that the condition $0>g_{\varphi (0)}(k)\geq -b$
means that $a>c$, or, equivalently, $a-c \geq 1$, and so
$g_{\varphi(0)}(k)=a(-b)+bc=-b(a-c)\geq -b$. Hence $c=a-1$, and so
$g_{\varphi^i (0)}(k^{i}) = -b$, and thus $w$ is WAP along the
line $y=-b$.

\medskip

3. Suppose that $\Delta\frac{A-c}{-b}+t\geq A$. We need to prove
that $w$ is WAP.

Let $j$ be such that $g_{\varphi(1)}(j)=t$. Under these conditions
we will prove the following claim: If for some $m$ one has $w_m=1$
and $g_w(m)\geq A$, then $w_{km+j}=1$ and $g_w(k(m-1)+j)\geq A$.

Consider the occurrence of $1$ at the position $m$. By the
definition of the graphic of $w$, one has that $g_w(m-1)\geq A-c$,
and hence ${\rm pref}_{m-1}(w)$ contains at least
$\frac{c}{b+c}(m-1)-\frac{1}{b+c} (A-c)$ letters $0$ and at most
$\frac{b}{b+c}(m-1)+\frac{1}{b+c} (A-c)$ letters $1$. So, for the
image of this prefix one has $g_w(k(m-1))\geq \Delta
\frac{A-c}{-b}$. Since $w_m=1$, one has $w[k(m-1)+1,
km]=\varphi(1)$. Then $g_w(k(m-1)+j)=g_w(k(m-1))+t\geq \Delta
\frac{A-c}{-b}+t$, and we have $\Delta\frac{A-c}{-b}+t\geq A$, and
so $g_w(k(m-1)+j)\geq A$. The claim is proved.

Now consider the occurrence of $1$ corresponding to the value $A$
defined in the theorem, i. e., we consider $w_i=1$ such that
$g_w(i)=A$. Applying the claim we just proved to $m=i$ we have
$w_{k(i-1)+j}=1$, $g_w(k(i-1)+j)\geq A$. Now we can apply the
claim to $m=k(i-1)+j$ and get that $w_{k(k(i-1)+j)+j}=1$,
$g_w(k(k(i-1)+j))\geq A$. Continuing this line of reasoning, one
gets infinitely many positions $n$ for which $g_w(n)\geq A$. On
the other hand, it is easy to see that $g_w(k^l)<0$ for all
integers $l$. So, $w$ is WAP along one of the lines $y=C$, $A-\max
(b,c)+1 \leq C \leq \max (b,c)-1$. Additional $\pm \max (b,c)$ are
taken to guarantee integer points, since the graphic "jumps" by
$b$ and $c$.

\medskip

Now suppose that $\Delta\frac{A-c}{-b}+t < A$. We need to prove
that $w$ is not WAP.

Let $j$ be such that $g_{\varphi(1)}(j)=t$. Under these conditions
we prove the following claim: If for all $m$ in a prefix of $w$ of
length $N$ such that $w_m=1$ one has $g_w(m) \leq A$, then for all
$N+1 \leq l \leq Nk$ such that $w_{l}=1$ we have $g_w(l) <
\max_{m} \{ g_w(m) | 1\leq m\leq N, w_m=1\}$, or, equivalently,
$g_w(l) \leq \max_{m}\{ g_w(m)-1 | 1\leq m\leq N, w_m=1\}$.
Roughly speaking, the claim says that maximal values are
decreasing. The claim is proved in a similar way as the previous
claim, so we omit the proof.

Now consider occurrences of $1$ from $\varphi(0)$, i. e., we
consider $w_i=1$ such that $1\leq i \leq k$. By the conditions of
the part 3 of the theorem we have $g_w(i) \leq A$. Applying the
latter claim to $m=i$ we have that for all occurrences $l$ of $1$
in $w[k+1, k^2]$ it holds $g_w(l) \leq A-1$. By the definition of
the graphic $g_w$, maximal values are attained immediately after
the occurrences of $1$-s, so we actually have $g_w(l) \leq A-1$
for all $k+1\leq l\leq k^2$. Continuing this line of reasoning, we
get that for $k^n+1 \leq i \leq k^{n+1}$
 it holds $g_w(l) \leq A-n$. So, the word $w$ is not WAP (since $w$ can be WAP only along
 horizontal lines).

\end{proof}

Now we are going to show that a fixed point of a uniform morphism
is bounded WAP iff it is abelian periodic. This is probably known
or follows from some general characterizations of balance of
morphic words (e. g., \cite {a}), but we anyway provide a short
combinatorial proof to be self-contained.

\begin{theorem}
Let $w$ be a fixed point of binary $k$-uniform morphism $\varphi$.
The following are equivalent:

\noindent 1. $w$ is bounded WAP

\noindent 2. $w$ is abelian periodic

\noindent 3. $\varphi(0)\sim_{ab}\varphi(1)$ or $k$ is odd and
$\varphi (0) = (01)^{\frac{k-1}{2}} 0$, $\varphi (1) =
(10)^{\frac{k-1}{2}} 1$.
\end{theorem}

\begin{proof} We prove the theorem in the following way. Starting with bounded
WAP word $w$, we step by step restrict the form of $w$ and prove
that the morphism should satisfy either
$\varphi(0)\sim_{ab}\varphi(1)$ or $k$ is odd and $\varphi (0) =
(01)^{\frac{k-1}{2}} 0$, $\varphi (1) = (10)^{\frac{k-1}{2}} 1$.
These conditions clearly imply abelian periodicity, and abelian
periodicity implies bounded WAP. So, we actually prove $1
\Rightarrow 3 \Rightarrow 2 \Rightarrow 1 $, and the only
implication to be proved is $1 \Rightarrow 3$.

Suppose that $w$ is bounded WAP and $\varphi(0)$ is not abelian
equivalent to $\varphi(1)$, i. e., $a\neq c$. Without loss of
generality we may assume that the fixed point starts in $0$ and
that $a>c$. If $a<c$, we consider a morphism $\varphi^2$, so that
one has $g_{\varphi^2(0)}\leq 0$. We will prove that either the
fixed point is not of bounded width or the morphism is of the form
$\varphi (0) = (01)^{\frac{k-1}{2}} 0$, $\varphi (1) =
(10)^{\frac{k-1}{2}} 1$, $k$ odd.

In the proof we will use the following notation. For a factor $u$
of $w$ such that $\rho_0(u)>\rho_0(w)$, we say that $u$ has $m$
\emph{extra} $0$'s, if $\frac{|u|_0-m}{|u|-m}=\rho_0(w)$. In other
words, deleting $m$ letters $0$ from $u$ gives a word with
frequency $\rho_0(w)$. We also admit non-integer values of $m$. E.
g., if $\rho_0(w)=\frac{1}{3}$ and $u =01$, then $u$ has $\frac{1}{2}$ extra $0$'s.

Suppose $a>c+1$. In this case $\varphi^i(0)$ contains $(a-c)^i$
extra zeros. Since $(a-c)^i$ increases as $i$ increases, $w$ is
not of bounded width. So, the fixed point is not bounded WAP in
this case, and hence for bounded WAP one should have $a=c+1$.

Suppose that $\varphi(0)$ has a prefix $x$ with more than one
extra zero. Without loss of generality we assume that $x$ ends in
$0$, otherwise we may take a smaller prefix. So, $x=x'0$, and $x'$
has $m>0$ extra $0$-s. It is not difficult to show that under
condition $a=c+1$ the image $\varphi(x')$ also contains $m$ extra
$0$. An image of $x$ starts in $\varphi(x')x'0$. An image of this
word starts in $\varphi^2(x')\varphi(x')x'0$. Continuing taking
images, we get that for every $i$ the word $w$ has a prefix of the
form $\varphi^i(x')\varphi^{i-1}(x') \dots \varphi(x')x'0$. This
word contains $(i+1)m+1$ extra $0$-s, and this amount grows as $i$
grows. Hence $w$ word is not of bounded width, a contradiction.
So, we have that every prefix of $\varphi(0)$ has at most one
extra $0$, in particular, $\varphi(0)$ starts in $01$.

In a similar way we show that every suffix of $\varphi(0)$ has at
most one extra $0$. The only difference is we obtain a series of
factors (not prefixes) of $w$ with growing amount of extra $0$-s.

Now consider an occurrence of $0$ in $\varphi(0)$, i. e., $w_j=0$,
$1\leq j\leq k$. Due to what we just proved $\rho_0({\rm
pref}_{j-1}(\varphi(0))\geq \rho_0(w)$, and $\rho_0({\rm
suff}_{k-j}(\varphi(0))\geq \rho_0(w)$. Since $\varphi(0)$ has one
extra $0$, we have  $\rho_0({\rm
pref}_{j-1}(\varphi(0))=\rho_0({\rm suff}_{k-j}(\varphi(0)) =
\rho_0(w)$. So, $w_j$ can be equal to $0$ only if in the prefix
${\rm pref}_{j-1}(\varphi(0))$ the frequency of $0$ is the same as
in $w$.

On the other hand, if the frequencies in the ${\rm
pref}_{j-1}(\varphi(0))$ are the same as in $w$, then $w_j$ cannot
be equal to $1$. Suppose the converse; let $w_j=1$, then all
$w_l=1$, $l=j, \dots, k-1$, since by induction in all the prefixes
${\rm pref}_{l}(\varphi(0))$ the frequency of $0$ is less than
$\rho_0(w)$. So, in $\varphi(0)$ there will be less than one extra
$0$, a contradiction.

Thus, each time we have $\rho_0({\rm pref}_{j-1}(\varphi(0))=
\rho_0(w)$, we necessarily have $w_j=0$, otherwise $w_j=1$. Since
$|\varphi(0)|_0=a$, the frequency $\rho_0(w)$ is reached $a$
times, and $\varphi(0)$ consists of $a-1$ blocks with one $0$ and
with frequency $\rho_0(w)$, and one extra block $0$. Therefore,
$a-1$ divides $a-1+b$, i. e., $b=i(a-1)$ for some integer $i$. By
a similar argument applied for $\varphi(1)$ we get that $d-1$
divides $c-1$, which means $i(a-1)$ divides $a-1$. Hence $i=1$,
and so the matrix of the morphism is $\left( \begin{array}{cc}
a & a-1 \\
a-1 & a \end{array} \right)$. Combining this with the conditions
for positions of $0$ in $\varphi(0)$, we obtain
 $\varphi (0)
= (01)^{\frac{k-1}{2}} 0$, $\varphi (1) = (10)^{\frac{k-1}{2}} 1$.

\end{proof}


\section{On WAP of points in a shift orbit closure}

In this section we consider the following question: if a uniformly
recurrent word $w$ is WAP, what can we say about WAP of other
words with the language $F(w)$?

As a corollary from Theorem \ref{morph} we obtain the following
Proposition:

\begin{proposition} There exists a binary uniform morphism having
two infinite fixed points, such that one of them is WAP, and the
other one is not.
\end{proposition}

\begin{proof} Consider a morphism $\varphi: 0\to 0001, 1\to 1011$. Using Theorem \ref {morph} (3),
one gets that the fixed point starting from $0$ is not WAP. Using
Theorem \ref {morph} (1), one gets that the fixed point starting
from $1$ is WAP.
\end{proof}

\noindent \textbf{Remark.} In particular, this means that there
exist two words with same sets of factors such that one of them is
WAP while the other one is not.

\medskip In this section we need some more definitions.

Let $T: \Sigma^{\omega} \to \Sigma^{\omega}$ denote the
\emph{shift transformation} defined by $T: (x_n)_{n\in \omega} \to
(x_{n+1}) _{n\in \omega}$. The \emph{shift orbit} of an infinite
word $x\in \Sigma^{\omega}$ is the set $O(x) = \{T^i(x) | i\geq 0
\}$ and its \emph{closure} is given by $\overline{O}(x) = \{ y \in
\Sigma^{\omega} | {\rm Pref}(y) \subseteq {\rm Pref}(T^i(x))$,
where ${\rm Pref}(w)$ denotes the set of prefixes of a finite or
infinite word $w$. For a uniformly recurrent word $w$ any infinite
word $x$ in $\overline{O}(w)$ has the same set of factors as $w$.

We say that $w\in \Sigma^{\omega}$ has \emph{uniform frequency}
$\rho_a$ of a letter $a$, if every word in $\overline{O}(w)$ has
frequency $\rho_a$ of a letter $a$. In other words,  a letter
$a\in\Sigma$ has uniform frequency $\rho_a$ in $w$ if its minimal
frequency $\underline{\rho}_a = \lim_{n\to \infty}\inf_{x\in
F_n(w)} \frac{|x|_a}{|x|}$ is equal to its maximal frequency
$\overline{\rho}_a = \lim_{n\to \infty}\sup_{x\in F_n(w)}
\frac{|x|_a}{|x|}$, i. e. $\underline{\rho}_a=\overline{\rho}_a$.

\begin{theorem}
Let $w$ be an infinite binary uniformly recurrent word.

\noindent 1. If $w$ has irrational frequencies of letters, then
every word in its shift orbit closure is not WAP.

\noindent 2. If $w$ does not have uniform frequencies of letters,
then there is a point in a shift orbit closure of $w$ which is
WAP.

\noindent 3. If $w$ has uniform rational frequencies of letters,
then there is a point in a shift orbit closure of $w$ which is
WAP.

\noindent 4. There exists a non-balanced word $w$ with uniform
rational frequencies of letters, such that every point in a shift
orbit closure of $w$ is WAP.

\end{theorem}

\begin{proof}

\noindent 1. Follows from Proposition \ref{freq} (2).

\medskip

\noindent 2. Follows from Proposition \ref{freq} (3).

\medskip

\noindent 3. In the proof we use the notion of a return word. For
$u\in F(w)$, let $n_1<n_2<\dots$ be all integers $n_i$ such that
$u=w_{n_i}\dots w_{n_{i}+|u|-1}$. Then the word $w_{n_i}\dots
w_{n_{i+1}-1}$ is a \emph{return word} (or briefly \emph{return})
of $u$ in $w$ \cite {durand}, \cite {hz}, \cite {pz}.

We now build a WAP word $u$ from $\overline{O}(w)$. Start with any
factor $u_1$ of $w$, e. g. with a letter. Without loss of
generality assume that $\rho_{0}(u_1)\geq \rho_0(w)$. Consider
factorization of $w$ into first returns to $u_1$: $w = v^1_1 v^1_2
\dots v^1_i \dots$, so that $v^1_i$ is a return to $u_1$ for
$i>1$. Then there exists $i_1>1$ satisfying
$\rho_0(v_{i_1}^1)\geq\rho_0$. Suppose the converse, i. e., for
all $i>1$ $\rho_0(v_i^1)<\rho_0$. Due to uniform recurrence, the
lengths of $v_i^1$ are uniformly bounded, 
and hence $\rho_0(w)<\rho_0$, a contradiction. Take
$u_2=v^1_{i_1}$, so $u_1={\rm pref} (u_2)$. Now consider a
factorization of $w$ into first returns to $u_2$: $w = v^2_1 v^2_2
\dots v^2_i \dots$. Then there exists $i_2>1$ satisfying
$\rho_0(v_{i_2}^2)\leq\rho_0$, take $u_3=v^2_{i_2}$. Continuing
this line of reasoning to infinity, we build a word
$u=\lim_{n\to\infty}u_i$, such that $\rho_0(u_{2i})\geq\rho_0$,
$\rho_0(u_{2i+1})\leq\rho_0$. So, the graphic of $w$ with vectors
$\textbf{v}_0=(1,-1)$ and $\textbf{v}_0=(1,-1)$ intersects the line $y=\rho_0 x$
infinitely many times. Since $\rho_0$ is rational, by a pigeonhole
principle the graphic intersects in integer points infinitely many
times one of finite number (actually, a denominator of $\rho_0$)
of lines parallel to $y=\rho_0 x$. It follows that $u$ is WAP with
frequency $\rho_0$, and by construction $u\in \overline{O}(w)$.

\medskip

\noindent 4. Due to space limitations, we omit the proof of this
item.

\end{proof}

\end{document}